\documentclass[11pt]{amsproc}

\usepackage{amsmath}
\usepackage{amsthm}
\usepackage{amssymb}
\usepackage{mathtools}
\usepackage{mathrsfs}
\usepackage[most]{tcolorbox}
\usepackage{fullpage}
\usepackage[hypertexnames=false]{hyperref}
\usepackage{thmtools}
\usepackage{enumitem}
\usepackage{url}
\usepackage{physics}

\usepackage[capitalize,nameinlink,noabbrev,nosort]{cleveref}
\hypersetup{
	colorlinks=true,       
	linkcolor=olive,         
	citecolor=olive,        
	filecolor=olive,      
	urlcolor=olive,           
}

\usepackage[abbrev]{amsrefs}

\newtheorem{theoremcounter}{Theorem Counter}[section]

\theoremstyle{remark}
\newtheorem{remark}[theoremcounter]{Remark}

\theoremstyle{definition}

\newtheorem{example}[theoremcounter]{Example}

\theoremstyle{plain}
\newtheorem{lemma}[theoremcounter]{Lemma}
\newtheorem{proposition}[theoremcounter]{Proposition}
\newtheorem{corollary}[theoremcounter]{Corollary}
\newtheorem{conjecture}[theoremcounter]{Conjecture}
\newtheorem{theorem}[theoremcounter]{Theorem}
\newtheorem{question}[theoremcounter]{Question}

\numberwithin{equation}{section}

\newcommand{\bbC}{\mathbb{C}}
\newcommand{\bbZ}{\mathbb{Z}}
\newcommand{\bbQ}{\mathbb{Q}}

\newcommand{\calA}{\mathcal{A}}
\newcommand{\calC}{\mathcal{C}}
\newcommand{\calP}{\mathcal{P}}

\newcommand{\scrB}{\mathscr{B}}
\newcommand{\scrD}{\mathscr{D}}
\newcommand{\scrE}{\mathscr{E}}
\newcommand{\scrF}{\mathscr{F}}

\newcommand{\bsp}{\boldsymbol{p}}

\DeclareMathOperator{\Gal}{Gal}
\DeclareMathOperator{\ord}{ord}

\begin{document}

\title[]{Some results on naive transcendence in the ring of integers modulo infinitely large primes}

\author[]{Toshiki Matsusaka}
\address{Faculty of Mathematics, Kyushu University, Motooka 744, Nishi-ku, Fukuoka 819-0395, Japan}
\email{matsusaka@math.kyushu-u.ac.jp}

\author[]{Shin-ichiro Seki}
\address{Nagahama Institute of Bio-Science and Technology, 1266, Tamura, Nagahama, Shiga, 526-0829, Japan}
\email{s\_seki@nagahama-i-bio.ac.jp}



\maketitle
\begin{abstract}
This paper presents various transcendence results in the ring of integers modulo infinitely large primes $\mathcal{A}$.
In the ring $\mathcal{A}$, one can consider two notions of transcendence.
One is based on the notion of finite algebraic numbers introduced by Rosen, while the other is transcendence in the naive sense.
It is known that transcendence in the latter sense automatically implies transcendence in the former sense.
In this paper, we strengthen results of Anzawa--Funakura and Luca--Zudilin by removing some of their assumptions and, in some cases, upgrading them to statements of naive transcendence.
We also present several examples of naive transcendental numbers that do not seem to have appeared previously in the literature.
Although we are not able to establish naive transcendence for certain numbers, we prove the irrationality of numbers such as $\log_{\calA}(2)$ under the ABC conjecture.
\end{abstract}
\tableofcontents
\section{Introduction}
In this paper, we investigate transcendental number theory in the following ring:
\[
\calA\coloneqq\Biggl(\prod_{p:\text{ prime}}\bbZ/p\bbZ\Biggr) \Bigg/ \Biggl(\bigoplus_{p:\text{ prime}}\bbZ/p\bbZ\Biggr).
\]
This is called the \emph{ring of integers modulo infinitely large primes} (see, for example, \cite{Jarossay2020}).
In other literature (e.g., \cite{KanekoZagier}), it is also referred to as the \emph{poor man's ad\`ele ring}.
To the best of the authors’ knowledge, the earliest appearance in the literature is in a paper by Ax \cite{Ax1968}.

Let $\calP_{\calA}^0$ denotes the set of all finite algebraic numbers defined by Rosen in \cite{Rosen2020}.
Let $\calC_{\calA}$ be the integral closure of $\bbQ$ in $\calA$.
Note that $\calP_{\calA}^0$ and $\calC_{\calA}$ are both $\bbQ$-algebras and the following inclusions hold: $\bbQ\subsetneq\calP_{\calA}^0\subsetneq\calC_{\calA}\subsetneq\calA$.
The set $\calP_{\calA}^0$ is countable, while the set $\calC_{\calA}$ is uncountable.
In $\calA$, there are two possible definitions of transcendence: one given by not belonging to $\calP_{\calA}^0$, and the other by not belonging to $\calC_{\calA}$.
Although Rosen’s definition, in which algebraic elements are defined as those contained in $\calP_{\calA}^0$, gives the correct notion of algebraicity corresponding to $\overline{\bbQ}\subsetneq\bbC$, proving that an element does not belong to $\calC_{\calA}$ is, whenever this is actually the case, stronger than merely showing that it does not belong to $\calP^0_{\calA}$.
When it is necessary to emphasize transcendence in the sense of not belonging to $\calC_{\calA}$, we refer to it as \emph{naive transcendence}.

Since there are still relatively few papers in this area, we begin by mentioning a point that can easily be misunderstood.
For any polynomial $f(x)\in\bbQ[x]$ that is irreducible over $\bbQ$ and has degree at least two, there exist infinitely many primes $p$ for which the reduction of $f(x)$ modulo $p$ has no linear factor.
In other words, for any $\alpha\in\calA$, we have $f(\alpha)\neq 0$.
This may give the mistaken impression that $\bbQ=\calC_{\calA}$.
In fact, however, since $\calA$ has many zero divisors, there are in general many roots of reducible polynomials lying in $\calA\setminus\bbQ$.
See also \cite[Non-Example~10]{KanekoZagier}.

As for transcendence in $\calA$, there are the following three previous works.
The first work is due to Anzawa and Funakura.
The sequence $(F_n(q))_n$ is the \emph{$q$-Fibonacci sequence} defined by $F_0(q)=0$, $F_1(q)=1$, and the recurrence relation $F_{n+2}(q)=F_{n+1}(q)+q^nF_n(q)$, and was introduced by Schur in \cite{Schur1917}.
\begin{theorem}[{Anzawa--Funakura \cite[Theorem~1.2]{AnzawaFunakura2024}}]\label{thm:AnzawaFunakura}
Let $\scrF(q)\coloneqq(F_p(q)\bmod{p})_p$.
For a square-free integer $q>1$, $\scrF(q)\in\calA\setminus\calC_{\calA}$ under the generalized Riemann hypothesis.
\end{theorem}
Note that, in the case where $q=1$, $\scrF(1)\in\calP^0_{\calA}\setminus\bbQ$ is a finite algebraic number.

In contrast, Luca and Zudilin succeeded in removing the assumptions that $q$ is square-free and that the generalized Riemann hypothesis holds.
However, unlike Anzawa and Funakura, they did not consider transcendence in the naive sense.
\begin{theorem}[{Luca--Zudilin \cite[Theorem~1.3]{LucaZudilin2025}}]
For every integer $q>1$, $\scrF(q)\in\calA\setminus\calP^0_{\calA}$.
\end{theorem}
They then proved the transcendence of elements naturally defined from the traces of Frobenius of elliptic curves.
\begin{theorem}[{Luca--Zudilin \cite[Theorem~1]{LucaZudilin}}]
Let $E$ be an elliptic curve defined over $\bbQ$.
For a prime $p$ of good reduction for $E$, define $a_p(E)\coloneqq p+1-\#E(\mathbb{F}_p)$ to be the trace of Frobenius.
Let $\alpha(E)\coloneqq(a_p(E)\bmod{p})_p\in\calA$.
If $E$ does not have complex multiplication, then $\alpha(E) \in \calA \setminus \mathcal{P}^0_{\mathcal{A}}$.
\end{theorem}
As noted in \cite[Remark~1]{LucaZudilin}, an anonymous referee of \cite{LucaZudilin} pointed out that one can in fact prove that $\alpha(E)\in\calA\setminus\mathcal{C}_{\mathcal{A}}$ for $E/\bbQ$ without CM.

The first aim of this paper is to remove the assumption that the elliptic curve does not have CM, and to establish naive transcendence in the results of Luca and Zudilin.
\begin{theorem}[\cref{thm:refined_AnzawaFunakura} and \cref{thm:Frobenius_traces}]
\
\begin{enumerate}[label=\textup{(\roman*)}]
\item For every integer $q>1$, $\scrF(q)\in\calA\setminus\calC_{\calA}$.
\item Let $E$ be an elliptic curve defined over $\bbQ$. Then $\alpha(E)\in\calA\setminus\calC_{\calA}$.
\end{enumerate}
\end{theorem}
Although the results are strengthened, we emphasize that the ideas of the proofs are essentially due to Luca and Zudilin.
Their proofs were written as proofs that the elements in question do not belong to $\calP^0_{\calA}$.
In fact, although at first sight their arguments seem to depend on Rosen's formulation of finite algebraic numbers using Frobenius elements, they had the potential to be extended to proofs that the elements in question do not belong to $\calC_{\calA}$.
By axiomatizing those arguments as criteria, we obtain several transcendental numbers that were not explicitly given in the previous literature, as the second aim of this paper.

Let $D_n(q)$, $B_n$, and $E_n$ denote the $n$th \emph{Bressoud polynomial}, \emph{Bernoulli number}, and \emph{Euler number}, respectively (see \cref{sec:q-Fibonacci} and \cref{sec:Bernoulli} for their definitions).
\begin{theorem}[\cref{thm:Bressoud_transcence}, \cref{thm:scrB}, and \cref{thm:scrE}]\ 
\begin{enumerate}[label=\textup{(\roman*)}]
\item For every integer $q>1$, $\scrD(q)\coloneqq(D_{p-1}(q) \bmod{p})_p \in \calA\setminus\calC_{\calA}$.
\item $\scrB\coloneqq(B_{\frac{p+1}{2}}\bmod{p})_p\in\calA\setminus\calC_{\calA}$.
\item $\scrE\coloneqq(E_{\frac{p-1}{2}}\bmod{p})_p\in\calA\setminus\calC_{\calA}$.
\end{enumerate}
\end{theorem}
The number $\scrD(q)$ is an analogue of $\scrF(q)$, while $\scrB$ and $\scrE$ are, at least apparently, of a different kind from the numbers treated in the previous works.

In addition, we give examples of naive transcendental elements such as $(\lfloor \log p\rfloor \bmod{p})_p$ and $(\lfloor \sqrt{p}\rfloor \bmod{p})_p$ in \cref{sec:examples}, and also discuss the transcendence of $\pi(\bsp)\coloneqq(\pi(p)\bmod{p})_p$ in \cref{sec:pip} (the question of naive transcendence remains open).
Here, $\lfloor X\rfloor$ denotes the greatest integer less than or equal to $X$, and $\pi(X)$ denotes the prime counting function. 

It is easy to prove that the example given by Anzawa and Funakura in~\cite{AnzawaFunakura2024} and $(\lfloor \log p\rfloor \bmod{p})_p$ are algebraically independent.
However, it is a difficult problem to determine whether numbers defined more naturally from arithmetic objects are algebraically independent.
These topics are discussed in \cref{sec:alg_indep}.

Finally, we discuss logarithmic values in $\calA$ in \cref{sec:log}.
Let $p$ be a prime and $\bbZ_{(p)}$ be the localization of $\bbZ$ with respect to the prime ideal $(p)$.
For $\alpha\in\bbZ_{(p)}^{\times}$, we define the \emph{Fermat quotient} of $\alpha$ by
\[
q_p(\alpha)\coloneqq \frac{\alpha^{p-1}-1}{p}\in\bbZ_{(p)}.
\]
Then the \emph{logarithm function} $\log_{\calA}\colon\bbQ^{\times}\to\calA$ is defined by $\log_{\calA}(\alpha)\coloneqq (q_p(\alpha)\bmod{p})_p$.
This is a group homomorphism: $\log_{\calA}(\alpha\beta)=\log_{\calA}(\alpha)+\log_{\calA}(\beta)$.
The values of $\log_{\calA}$ appear in number theoretic contexts (cf.~\cite{KanekoMatsusakaSeki2025}, \cite{Seki2024}) and the following conditional result is known concerning their nonvanishing.
\begin{theorem}[Silverman \cite{Silverman1988}]\label{thm:Silverman-0}
Let $\alpha\not\in\{+1,-1\}$ be a nonzero rational number.
Then, $\log_\calA(\alpha)\neq 0$ under the ABC conjecture.
\end{theorem}
As stated in \cite[Section~3]{LucaZudilin2025}, proving the irrationality of logarithmic values is ``perhaps the first step to do before looking into the irrationality of $Z_{\calA}(k)$ for $k > 1$ odd'' (see \cref{sec:Bernoulli} for $Z_{\calA}(k)$).
In this paper, we complete this first step, again under the ABC conjecture.
\begin{theorem}[\cref{cor:Silverman-irr}]
Let $\alpha\not\in\{+1,-1\}$ be a nonzero rational number.
Then, $\log_\calA(\alpha) \in \calA\setminus\bbQ$ under the ABC conjecture.
\end{theorem}
We emphasize that the assertion $\log_{\calA}(\alpha)\in\calA\setminus\bbQ^{\times}$ is unconditional (see \cref{thm:unconditional_Silverman}).
We were not able to extend the present proof to a proof of transcendence.

Recently, a positive-characteristic analogue of finite algebraic numbers were introduced in \cite{MatsuzukiSakamotoUeki}.
An analogue of the example of Anzawa--Funakura in positive characteristic is also treated there.
Investigating similar phenomena in positive characteristic for the results of the present paper would be an interesting topic for future research.
\subsection*{Acknowledgements}
The authors are grateful to Dr.~Takumi Anzawa and Prof.~Jun Ueki for helpful discussions on notation at the draft stage.
The second author has been interested in the topic of this paper since his student days, before receiving his degree in 2017, and has continued to think about it ever since.
He would like to thank Prof.~Kenji Sakugawa and Prof.~Shuji Yamamoto for the many discussions since that time.
The authors were supported by JSPS KAKENHI (JP24K16901 and JP26K06734, respectively).
\section{Criteria and examples}\label{sec:examples}
In this section, we collect four criteria for naive transcendence and present several artificial transcendental numbers.

Although this notation may not be standard, in this paper we distinguish the following two meanings according to whether parentheses are attached to $\mathrm{mod}$ $p$.
An element of $\bbZ/p\bbZ$ is written as $a\bmod{p}\in\bbZ/p\bbZ$, where $a$ is an integer representative of the residue class.
On the other hand, for integers $a$ and $b$, the congruence relation is written in the form $a\equiv b\pmod{p}$.

Anzawa and Funakura gave the following simple criterion for naive transcendence and used it to prove \cref{thm:AnzawaFunakura}.
\begin{lemma}[{Anzawa--Funakura's criterion \cite[Proposition~3.7]{AnzawaFunakura2024}}]\label{lem:AnzawaFunakura}
Let $\alpha=(a_p)_p\in\calA$.
Let $(b_n)_n$ be a strictly increasing sequence of positive integers.
If, for every $n$, there exist infinitely many primes $p$ such that $a_p=b_n\bmod{p}$, then $\alpha\in\calA\setminus\calC_{\calA}$.
\end{lemma}
\begin{example}[{\cite[Example~3.8]{AnzawaFunakura2024}}]
$(t_{\pi(p)}\bmod{p})_p\in\calA\setminus\calC_{\calA}$.
Here, $\pi(X)$ denotes the prime counting function, and $(t_n)_n=(1,1,2,1,2,3,1,2,3,4,\dots)$.
This appears to be the first explicit example of a naive transcendental number in the literature, just as Liouville's example was the first explicit example in the classical setting of transcendental numbers.
\end{example}
There is also the following simple criterion, which is different from Anzawa--Funakura's one.
\begin{lemma}\label{lem:old}
If a sequence of integers $(a_p)_p$ satisfies $a_p\to\infty$ and $a_p^d=o(p)$ as $p\to\infty$ for every positive integer $d$, then $(a_p\bmod{p})_p\in\calA\setminus\calC_{\calA}$.
\end{lemma}
\begin{proof}
Let $f(x)$ be a polynomial of degree $d\geq1$ with integer coefficients.
Suppose that $f(a_p)\equiv0\pmod{p}$ holds for all sufficiently large primes $p$.
Since $f(n)=O(n^d)$ for positive integers $n$, it follows that $|f(a_p)|<p$ for all sufficiently large $p$.
The condition $f(a_p)\equiv0\pmod{p}$ then implies that $f(a_p) = 0$.
This contradicts the fact that $a_p\to\infty$.
\end{proof}
\begin{example}
$(\lfloor\log p\rfloor\bmod{p})_p\in\calA\setminus\calC_{\calA}$.
Here, $\lfloor X\rfloor$ denotes the greatest integer less than or equal to $X$.
\end{example}
A criterion that applies even when the growth is not as slow as required in \cref{lem:old} can be extracted from the proof of the main theorem in the second paper of Luca and Zudilin \cite{LucaZudilin} (see also \cite[Remark~1]{LucaZudilin}).
\begin{lemma}[Luca--Zudilin's second criterion]\label{lem:LZ-criterion2}
Let $\alpha = (a_p)_p \in \mathcal{A}$. Let $S$ be an infinite set of primes, and let $(b_p)_{p \in S}$ be a sequence of integers satisfying $b_p\to\infty$, and $b_p = O(p^{\epsilon})$ as $p \to \infty$ for some $0 < \epsilon < 1$.
Suppose that $a_p=b_p \bmod{p}$ for all $p \in S$. 
Assume further that there exists $\epsilon'$ with $\epsilon < \epsilon' < 1$ such that $\#\{ p \le X : p \in S \} \gg X^{\epsilon'}$ for all sufficiently large $X$. 
Then $\alpha \in \mathcal{A} \setminus \mathcal{C}_{\mathcal{A}}$.
\end{lemma}
\begin{proof}
Let $f(x)$ be a nonzero polynomial with integer coefficients, and suppose that $f(\alpha)=0$.
Then $f(b_p)\equiv 0 \pmod{p}$ holds for all but finitely many $p\in S$.
Let $R$ denote the set of all integer roots of $f(x)$.
It is clear that $R$ is a finite set.
For sufficiently large real numbers $X$, define
\[
S(X)\coloneqq\{p\leq X : p\in S,\ b_p\not\in R,\ p\mid f(b_p)\}.
\]
Then, by the assumption, $\#S(X)\gg X^{\epsilon'}$ holds.

For each $p \in S(X)$, by the assumption, we have $b_p = O(X^{\epsilon})$.
Hence, the number of integers of the form $f(b_p)$, as $p$ ranges over $S(X)$, is at most $O(X^{\epsilon})$.
For each such integer (note that $f(b_p)\neq 0$ since $b_p \not\in R$), the number of its prime divisors is at most $O(\log X)$.
This follows from the fact that, for a positive integer $n$, the number of its prime divisors satisfies $\omega(n)=O(\log n)$, together with the estimate $f(b_p)=O(X^{O(1)})$. 
Since every element of $S(X)$ coincides with one of such prime divisors, it follows that $\#S(X)=O(X^{\epsilon}\log X)$.
Therefore, we have $X^{\epsilon'} \ll X^{\epsilon}\log X$.
Since $\epsilon<\epsilon'$, this yields a contradiction for sufficiently large $X$.
\end{proof}
By this criterion, we obtain an example of transcendental numbers as follows, which cannot be derived from \cref{lem:AnzawaFunakura} and \cref{lem:old}.
\begin{example}
$(\lfloor\sqrt{p}\rfloor\bmod{p})_p\in\calA\setminus\calC_{\calA}$.
\end{example}
The following criterion can be extracted from the proof of the main theorem in the first paper of Luca and Zudilin \cite{LucaZudilin2025}.
For a prime $p$ and $q\bmod{p} \in (\bbZ/p\bbZ)^\times$, we denote by $\ord_p(q)$ the order of $q\bmod{p}$, and set $I_p(q) \coloneqq [(\bbZ/p\bbZ)^\times : \langle q\bmod{p} \rangle] = (p-1)/\ord_p(q)$.
\begin{lemma}[Luca--Zudilin's first criterion]\label{lem:LZ-criterion1}
Let $\alpha = (a_p)_p \in \calA$.
Let $q>1$ and $d\geq 1$ be integers.
For $1 \le j \le d$, let $(b_{j,n})_n$ be sequences of integers satisfying $|b_{j,n}| \to \infty$ and $\log |b_{j,n}| = O(n)$ as $n \to \infty$.
Let $(N,c)$ be a pair of coprime positive integers.
Suppose that, for all sufficiently large primes $p \equiv c \pmod{N}$, there exists an index $1 \le j \le d$ for which $a_p = b_{j, I_p(q)} \bmod{p}$.
Then $\alpha \in \calA \setminus \calC_\calA$.
\end{lemma}
\begin{proof}
Let $f(x)$ be a nonzero polynomial with integer coefficients, and suppose that $f(\alpha) = 0$.
Let $R$ denote the set of all integer roots of $f(x)$.
Choose a sufficiently large odd integer $r$, coprime to $N$, such that $b_{j,n} \not\in R$ for all $1 \le j \le d$ and all $n \ge r$. 
For sufficiently large real numbers $X$, define
\[
P(X) \coloneqq \left\{p \in [X,2X] : p \equiv 1\ (\mathrm{mod}\ r), \ p \equiv c\ (\mathrm{mod}\ N), \ \ord_p(q) \mid \frac{p-1}{r} \right\}.
\]
By our assumptions, for every $p\in P(X)$, there exists an index $1\le j\le d$ such that $a_p=b_{j,I_p(q)}\bmod{p}$ and
\begin{equation}\label{eq:LZ1div}
p \mid f(b_{j, I_p(q)}) \neq 0.
\end{equation}
Here the nonvanishing follows from the inequality $I_p(q)=(p-1)/\ord_p(q)\ge r$ and the choice of $r$.

We first prove that
\[
\#P(X)\gg\frac{X}{\log X}.
\]
Write $q=q_0^k$, where $k$ is a positive integer and $q_0$ is not a perfect power.
Set $m\coloneqq r/\gcd(r,k)$ and $K\coloneqq\bbQ(\zeta_{rN},\!\!\sqrt[m]{q_0})$.
Here, $\zeta_{rN}$ is a primitive $rN$-th root of unity.
By the Chinese remainder theorem, we may choose an integer $a$ satisfying $a\equiv 1 \pmod{r}$ and $a\equiv c \pmod{N}$.
Let $\sigma_a\in\Gal(\bbQ(\zeta_{rN})/\bbQ)$ be the element determined by $\sigma_a(\zeta_{rN})=\zeta_{rN}^a$.
Let $p$ be a prime unramified in $K$, and let $\mathfrak{P}$ be any prime ideal of $\mathcal{O}_K$ lying above $p$.
Write $\varphi_{\mathfrak{P}}\in\Gal(K/\bbQ)$ for the Frobenius element at $\mathfrak{P}$.
Under the assumption $p\equiv 1 \pmod{r}$, the condition $\ord_p(q)\mid\frac{p-1}{r}$ is equivalent to $q^{\frac{p-1}{r}}\equiv 1\pmod{p}$, which is equivalent to $q_0^{\frac{p-1}{m}}\equiv 1\pmod{p}$, and this is further equivalent to $\varphi_{\mathfrak{P}}(\!\!\sqrt[m]{q_0})=\sqrt[m]{q_0}$.
Therefore, for sufficiently large $X$, we have
\[
P(X)=\left\{p\in [X,2X] : \varphi_{\mathfrak{P}}|_{\bbQ(\zeta_{rN})}=\sigma_a, \ \varphi_{\mathfrak{P}}(\!\!\sqrt[m]{q_0})=\sqrt[m]{q_0}\right\}.
\]
On the other hand, since $r$ is odd, $m$ is also odd, and $\bbQ(\zeta_{rN})\cap\bbQ(\!\!\sqrt[m]{q_0})=\bbQ$.
It follows that $x^m-q_0$ is irreducible over $\bbQ(\zeta_{rN})$.
Therefore,
\[
\{\tau\in\Gal(K/\bbQ) : \tau|_{\bbQ(\zeta_{rN})}=\sigma_a,\ \tau(\!\!\sqrt[m]{q_0})=\sqrt[m]{q_0}\}
\]
is nonempty, and in fact consists of a single element.
Hence the Chebotarev density theorem is applicable.

On the other hand, we decompose $P(X)$ into three subsets as follows, and estimate each of them separately:
\begin{align*}
P_1(X) &\coloneqq \{p \in P(X) : \ord_p(q) \le \sqrt{X}/\log X\},\\
P_2(X) &\coloneqq \{p \in P(X) : I_p(q) \le \sqrt{X}/\log X\},\\
P_3(X) &\coloneqq \{p \in P(X) : \ord_p(q), I_p(q) > \sqrt{X}/\log X\}.
\end{align*}

First, for $p \in P_1(X)$, we have $p \mid (q^{\ord_p(q)} - 1)$, and hence
\[
\prod_{p \in P_1(X)} p \le \prod_{n \le \sqrt{X}/\log X} (q^n-1) \le \exp \left(\log q \sum_{n \le \sqrt{X}/\log X} n \right) = \exp\left(O \left(\frac{X}{(\log X)^2}\right)\right).
\]
Note that this inequality argument does not hold when $q=1$.
Since $p\ge X$, it follows that
\[
\#P_1(X)\ll\frac{X}{(\log X)^3}.
\]

Second, the number of integers of the form $f(b_{j, I_p(q)})$, as $p$ ranges over $P_2(X)$, is at most $O(\sqrt{X}/\log X)$.
By the same argument as in the proof of \cref{lem:LZ-criterion2}, the number of prime divisors of each $f(b_{j, I_p(q)})$ is at most $O(\log |b_{j, I_p(q)}|) = O(\sqrt{X}/\log X)$.
Hence, by \eqref{eq:LZ1div}, we obtain
\[
\#P_2(X)\ll\frac{X}{(\log X)^2}.
\]

We estimate $\#P_3(X)$ by using Ford's result.
We use the notation $H(x,y,z)$ and $H(x,y,z;P_{\lambda})$ from \cite{Ford2008}.
If $p\in P_3(X)$, then $\ord_p(q)\mid p-1$ and
\[
\frac{\sqrt{X}}{\log X}<\ord_p(q)\leq 2\sqrt{X}\log X.
\]
Therefore,
\[
\#P_3(X)\leq H\left(2X,\frac{\sqrt{X}}{\log X}, 2\sqrt{X}\log X; P_{-1}\right).
\]
By \cite[Theorem~6]{Ford2008}, we have
\[
H\left(2X,\frac{\sqrt{X}}{\log X}, 2\sqrt{X}\log X; P_{-1}\right)\ll \frac{H\left(2X,\frac{\sqrt{X}}{\log X}, 2\sqrt{X}\log X\right)}{\log X}
\]
and by \cite[Theorem~1~(v)]{Ford2008}, we have
\[
H\left(2X,\frac{\sqrt{X}}{\log X}, 2\sqrt{X}\log X\right)\ll X\cdot u^{\delta}\left(\log\frac{2}{u}\right)^{-\frac{3}{2}},
\]
where
\[
u=\frac{\log z}{\log y}-1\sim \frac{4\log\log X}{\log X}
\]
and $\delta= 1-(1+\log\log 2)/\log 2=0.08607\dots$. 
Therefore, we have
\[
\#P_3(X)\ll\frac{X(\log\log X)^{O(1)}}{(\log X)^{1+\delta}}.
\]

Combining the above estimates, we obtain
\[
\frac{X}{\log X}\ll\#P(X)\le\#P_1(X)+\#P_2(X)+\#P_3(X)\ll\frac{X(\log \log X)^{O(1)}}{(\log X)^{1+\delta}},
\]
which yields a contradiction.
\end{proof}
The estimates for $\#P_1(X)$ and $\#P_3(X)$ are obtained by the same method as in \cite{LucaZudilin2025}.
On the other hand, for $\#P_2(X)$, the key point is to incorporate into the proof from the first paper of Luca--Zudilin \cite{LucaZudilin2025} the counting method introduced in their second paper \cite{LucaZudilin}.
\begin{example}
$(I_p(2)\bmod{p})_p\in\calA\setminus\calC_{\calA}$.
Note that $(I_p(2)\bmod{p})_p-1$ is a zero divisor if and only if the set of primes $p$ for which $2$ is a primitive root modulo $p$ is infinite, in other words, \emph{Artin's primitive root conjecture} holds for $2$.
\end{example}
\section{Polynomials related to Rogers--Ramanujan's identity}\label{sec:q-Fibonacci}
The $q$-Fibonacci sequence was originally introduced in connection with a finite form of the \emph{first Rogers--Ramanujan identity},
\[
F_{n+1}(q) = \sum_{k=0}^{\lfloor n/2 \rfloor} q^{k^2} \binom{n-k}{k}_q = \sum_{k \in \bbZ} (-1)^k q^{\frac{k(5k+1)}{2}} \binom{n}{\lfloor \frac{n+5k+1}{2}\rfloor}_q,
\]
whose limit, together with Jacobi's triple product, yields the identity
\[
\sum_{n=0}^\infty \frac{q^{n^2}}{(q)_n} = \prod_{n=0}^\infty \frac{1}{(1-q^{5n+1})(1-q^{5n+4})}.
\]
Here, $\binom{n}{k}_q \coloneqq \frac{(q)_n}{(q)_k (q)_{n-k}}$ denotes the \emph{$q$-binomial coefficient}, and $(q)_n \coloneqq (1-q)(1-q^2) \cdots (1-q^n)$ is the \emph{$q$-Pochhammer symbol}. 
In this section, we prove, via Luca–Zudilin's first criterion, that the naive transcendence of $\scrF(q)$ holds without assuming the generalized Riemann hypothesis. 
We also consider another such finite form of the same identity, given by the polynomial sequence $(D_n(q))_n$ introduced by Bressoud~\cite{Bressoud1981}. 
For background on finite forms, see Sills~\cite[Chapter 4]{Sills2018}.

Let $(F_n)_n$ be the \emph{Fibonacci sequence}.
Note that $F_n(1)=F_n$.
\begin{theorem}[{Anzawa--Furakura \cite[Theorem~1.1]{AnzawaFunakura2024}}]\label{thm:AnzawaFunakura_congruence}
For a nonzero rational number $q$ and a prime $p$ satisfying $q\in\bbZ_{(p)}^{\times}$ and $\ord_p(q)\not\equiv0\pmod{5}$, we have
\[
F_p(q)\equiv F_{I_p(q)+\left(\frac{\ord_p(q)}{5}\right)} \pmod{p},
\]
where $\left(\frac{\cdot}{\cdot}\right)$ is the Legendre symbol.
\end{theorem}
\begin{remark}
Anzawa and Funakura state the theorem excluding the case where $q-1\in p\bbZ_{(p)}\setminus\{0\}$, but it is easy to check that the theorem remains valid even in that case.
\end{remark}
\begin{theorem}\label{thm:refined_AnzawaFunakura}
For every integer $q > 1$, $\scrF(q) \in \calA \setminus \calC_\calA$.
\end{theorem}
\begin{proof}
By \cref{thm:AnzawaFunakura_congruence}, we see that \cref{lem:LZ-criterion1} applies to the sequences $b_{1,n}=F_{n-1}$ and $b_{2,n}=F_{n+1}$.
To ensure the condition $\ord_p(q) \not\equiv 0 \pmod{5}$, it suffices, for instance, to restrict to primes $p \equiv 4 \pmod{5}$. Indeed, if $5 \mid \ord_p(q)$, then $5 \mid (p-1)$, which is a contradiction.
The condition $\log|b_{j,n}|=O(n)$ follows from Binet's formula.
\end{proof}
We define the Bressoud polynomials $D_n(q)$ by the initial conditions $D_0(q) = 1, D_1(q) = 1+q$, together with the recurrence
\[
D_n(q) = (1+q-q^n+q^{2n-1}) D_{n-1}(q) - q(1-q^{n-1})D_{n-2}(q).
\]
Bressoud showed that these polynomials satisfy the following finite form of the first Rogers--Ramanujan identity:
\begin{align}\label{eq:Bre-identity}
D_n(q) = \sum_{k=0}^n q^{k^2} \binom{n}{k}_q = \sum_{k \in \bbZ} (-1)^k q^{\frac{k(5k+1)}{2}} \binom{2n}{n+2k}_q.
\end{align}
This provides a $q$-analogue of $2^n$. In this setting, the following holds.
\begin{proposition}\label{prop:Bressoud}
For $q\in\bbQ^\times$ and a prime $p$ satisfying $q\in\bbZ_{(p)}^{\times}$, we have 
\[
D_{p-1}(q) \equiv 2^{I_p(q)} \pmod{p}.
\]
\end{proposition}
\begin{proof}
In~\cite[Lemma~2.6]{AnzawaFunakura2024}, Anzawa and Funakura proved that
\[
\binom{p-1}{k}_q \equiv \begin{cases}
\binom{I_p(q)}{k/\mathrm{ord}_p(q)} \pmod{p} &\text{if } \mathrm{ord}_p(q) \mid k,\\
0 \pmod{p} &\text{otherwise},
\end{cases}
\]
under the additional assumption that $q-1 \in \bbZ_{(p)}^{\times}$.
However, it is easy to see that this also holds when this assumption is not satisfied.
Therefore, by \eqref{eq:Bre-identity}, it follows that
\[
D_{p-1}(q) = \sum_{k=0}^{p-1} q^{k^2} \binom{p-1}{k}_q \equiv \sum_{j=0}^{I_p(q)} \binom{I_p(q)}{j}=2^{I_p(q)} \pmod{p}.\qedhere
\]
\end{proof}
We consider $\scrD(q)\coloneqq(D_{p-1}(q)\bmod{p})_p$.
\begin{theorem}\label{thm:Bressoud_transcence}
For every integer $q > 1$, $\scrD(q)\in\calA\setminus\calC_\calA$.
\end{theorem}
\begin{proof}
Using the congruences established in \cref{prop:Bressoud}, it suffices to apply \cref{lem:LZ-criterion1} with the sequence $b_n = 2^n$ and any arithmetic progression you like.
\end{proof}
Note that $\scrD(1)=1$ by Fermat's little theorem.
\section{Frobenius traces of elliptic curves}\label{sec:Elliptic-Curve}
In this section, we prove the following theorem\footnote{Yuto Tsuruta has also independently obtained a result that overlaps with this theorem, using a different method.}.
\begin{theorem}\label{thm:Frobenius_traces}
Let $E$ be an elliptic curve over $\bbQ$.
Then, $\alpha(E)\coloneqq(a_p(E)\bmod{p})_p\in\calA\setminus\calC_{\calA}$.
\end{theorem}
As noted in \cite[Remark~1]{LucaZudilin}, the result is already known when $E$ does not have complex multiplication.
The proof given there uses the theorem of Kaneko--Elkies and Serre’s open image theorem, technically.
From the point of view of Luca--Zudilin's second criterion, however, it suffices to establish a statement such as the following lemma, which holds regardless of whether $E$ has complex multiplication, although the argument ultimately relies on the Sato--Tate distribution and still involves a case-by-case analysis.
\begin{lemma}\label{lem:ST-dist}
There exist constants $0<c_1<c_2\leq 1$ such that the set of primes $p$ satisfying $2c_1\sqrt{p}\leq a_p(E)\leq 2c_2\sqrt{p}$ has positive relative density.
\end{lemma}
\begin{proof}
In fact, this holds for any constants $0<c_1<c_2\leq 1$.
For an elliptic curve $E/\bbQ$ and an interval $I = [\alpha, \beta] \subset [0,\pi]$, it is known that
\[
\lim_{X \to \infty} \frac{\#\{p \le X : p\text{ is prime}, \theta_p \in I\}}{\#\{p \le X : p\text{ is prime}\}} = \begin{cases}
\displaystyle{\frac{2}{\pi} \int_I \sin^2\theta \mathrm{d} \theta} &\text{if } E \text{ is non-CM},\\
\displaystyle{\frac{\delta_I}{2} + \frac{\beta-\alpha}{2\pi}} &\text{if $E$ has CM},
\end{cases}
\]
where $a_p = 2\sqrt{p} \cos \theta_p$ and $\delta_I = 1$ if $\pi/2 \in I$ and $\delta_I = 0$ otherwise. 
The CM case is classical and goes back to Hecke and Deuring; see, for example, Sutherland's survey~\cite[Section~2.4]{Sutherland2019}.
The non-CM case was finally proved by Barnet-Lamb, Geraghty, Harris, and Taylor \cite{BGHT2011}.
Therefore, the set of primes $p$ satisfying $2c_1 \sqrt{p} \le a_p(E) \le 2c_2 \sqrt{p}$, that is, $c_1 \le \cos \theta_p \le c_2$, has positive relative density.
\end{proof}
\begin{proof}[Proof of \cref{thm:Frobenius_traces}]
Let $S$ be the set of primes of positive relative density obtained in \cref{lem:ST-dist}. 
For each $p \in S$, define $b_p \coloneqq a_p(E)$.
Then, $b_p \to \infty$ and $b_p = O(p^{\frac12})$.
Moreover, since $\# \{p \le X : p \in S\} \gg X/\log X$, it follows from \cref{lem:LZ-criterion2} that $\alpha(E) \in \calA \setminus \calC_\calA$.
\end{proof}
\section{Bernoulli and Euler numbers}\label{sec:Bernoulli}
Let $(B_n)_{n\geq 0}$ be the sequence of Bernoulli numbers defined by
\[
\frac{te^t}{e^t-1}=\sum_{n=0}^{\infty}\frac{B_n}{n!}t^n.
\]
Using the Bernoulli numbers, we define $Z_{\calA}(k)$ for integers $k\geq2$ as follows:
\[
Z_{\calA}(k)\coloneqq\left(\frac{B_{p-k}}{k}\bmod{p}\right)_p\in\calA.
\]
Recalling that $B_n = 0$ for odd integers $n\geq3$, we see that $Z_{\calA}(k)=0$ whenever $k$ is even.
On the other hand, it is conjectured that $Z_{\calA}(3), Z_{\calA}(5), Z_{\calA}(7), Z_{\calA}(9), \dots$ are algebraically independent over $\bbQ$ (cf.~\cite[Theorem~5.3]{Rosen}).
From the point of view of the Kaneko--Zagier conjecture (\cite[Main Conjecture]{KanekoZagier}), recalling that $Z_{\calA}(k)$ corresponds to $\zeta(k)\bmod{\pi^2}$, this conjecture may be regarded as corresponding to the conjecture that $\pi^2, \zeta(3), \zeta(5), \zeta(7), \zeta(9), \dots$ are algebraically independent over $\bbQ$.
It is natural to regard the counterpart of $\pi^2$ (or $2\pi i$) in $\calA$ as being $0$.

However, nothing is known about this conjecture.
In fact, there is not even a single value of $k$ for which $Z_{\calA}(k)\neq0$ has been proved.
Moreover, it has not even been shown that the set $\{k: Z_{\calA}(k)\neq0\}$ is nonempty.
This stands in contrast to the case of the Riemann zeta values, where Ap\'ery's theorem and Rivoal's theorem are known.
The numbers $Z_{\calA}(k)$ appear as values of \emph{finite multiple zeta values} defined by Kaneko and Zagier, for example, $\zeta_{\calA}^{}(1,2)=3Z_{\calA}(3)$.
Nevertheless, the existence of a nonzero finite multiple zeta value $\zeta_{\calA}^{}(k_1,\dots,k_r)$ has not been established at all, except for $\zeta_{\calA}^{}(\varnothing)=1$.
Assuming the infinitude of regular primes, one can barely show that, for each $k\neq1,2,4$, there exists a nonzero finite multiple zeta value of weight $k$ (\cite[Proposition~2~(i)]{Seki2024}).

Thus, the nonvanishing of $Z_{\calA}(k)$ for odd $k>1$ remains difficult and also outside the scope of this paper.
On the other hand, for
\[
\scrB\coloneqq(B_{\frac{p+1}{2}}\bmod{p})_p,
\]
which is likewise defined from Bernoulli numbers, one can prove its naive transcendence.
Note that $\scrB$ is a zero divisor, since $B_{\frac{p+1}{2}}=0$ whenever $p\equiv 1\mod{4}$.

Let $h(D)$ denote the class number of the quadratic number field of discriminant $D$.
In the case of imaginary quadratic fields, that is, when $D<0$, the famous theorem of Siegel \cite{Siegel1935} asserts that
\[
\lim_{D\to-\infty}\frac{\log h(D)}{\log|D|}=\frac{1}{2}.
\]
However, for our purposes, it is enough to use only the upper bound $h(D)\ll |D|^{\frac{1}{2}+\varepsilon}$ together with Heilbronn's result that $\displaystyle\lim_{D\to-\infty}h(D)=\infty$.

For a prime $p>3$ with $p\equiv 3 \pmod{4}$, the well-known congruence
\begin{equation}\label{eq:Cauchy}
-2B_{\frac{p+1}{2}}\equiv h(-p) \pmod{p}
\end{equation}
holds.
This congruence goes back at least to Cauchy, and is proved in \cite[Equation~(5.2)]{Carlitz1954}.
\begin{theorem}\label{thm:scrB}
$\scrB\in\calA\setminus\calC_{\calA}$.
\end{theorem}
\begin{proof}
Let $S$ be the set of all primes $p>3$ satisfying $p\equiv 3\pmod{4}$, and for each $p\in S$, define $b_p\coloneqq h(-p)$.
Then the facts that $b_p\to\infty$ and that $b_p=O(p^{\frac{1}{2}+\varepsilon})$ for every $\varepsilon>0$ follow from the classical results mentioned above. 
Since $\#\{p\leq X : p\in S\}\sim\frac{1}{2}\frac{X}{\log X}$, it follows from \eqref{eq:Cauchy} and \cref{lem:LZ-criterion2} that $-2\scrB\in\calA\setminus\calC_{\calA}$.
\end{proof}
Let $(E_n)_{n\geq 0}$ be the sequence of Euler numbers defined by
\[
\sech t=\frac{2}{e^t+e^{-t}}=\sum_{n=0}^{\infty}\frac{E_n}{n!}t^n,
\]
which is one of the analogues of Bernoulli numbers.

The Euler numbers give rise to questions analogous to those studied for the Bernoulli numbers.
For example, a prime $p$ that does not divide any of $E_2,E_4,\dots,E_{p-3}$ is called \emph{E-regular} (this is related to Fermat's last theorem \cite{Vandiver1940}), and it is an open problem whether there exist infinitely many such primes.
Corresponding to \cref{thm:scrB}, one can prove the naive transcendence of 
\[
\scrE\coloneqq(E_{\frac{p-1}{2}}\bmod{p})_p.
\]
Note that $\scrE$ is also a zero divisor, since $E_{\frac{p-1}{2}}=0$ whenever $p\equiv 3 \pmod{4}$.
\begin{theorem}\label{thm:scrE}
$\scrE\in\calA\setminus\calC_{\calA}$.
\end{theorem}
\begin{proof}
Let $S$ be the set of all primes $p$ satisfying $p\equiv 1\pmod{4}$, and for each $p\in S$, define $b_p\coloneqq h(-4p)$.
Then the facts that $b_p\to\infty$ and that $b_p=O(p^{\frac{1}{2}+\varepsilon})$ for every $\varepsilon>0$ follow from the classical results mentioned above.
Since $\#\{p\leq X : p\in S\}\sim\frac{1}{2}\frac{X}{\log X}$ and the congruence
\[
\frac{1}{2}E_{\frac{p-1}{2}}\equiv h(-4p)\pmod{p}
\]
holds (\cite[Equation~(1.2)]{Carlitz1954}), it follows from \cref{lem:LZ-criterion2} that $\frac{1}{2}\scrE\in\calA\setminus\calC_{\calA}$.
\end{proof}
Since $\scrB\scrE=0$, the numbers $\scrB$ and $\scrE$ are algebraically dependent (see \cref{sec:alg_indep} for the definition of algebraic independence and examples). 
Nevertheless, it is easy to see that the following holds.
\begin{proposition}\label{prop:BE-lin-indep}
The numbers $\scrB$ and $\scrE$ are linearly independent over $\bbQ$.
\end{proposition}
\begin{proof}
Assume that there exist integers $b$ and $e$ such that $b\scrB+e\scrE=0$.
Then, for all sufficiently large primes $p$, one has
\[
bB_{\frac{p+1}{2}}+eE_{\frac{p-1}{2}}\equiv 0\pmod{p}.
\]
If $p\equiv 1\pmod{4}$, then $B_{\frac{p+1}{2}}=0$ and $E_{\frac{p-1}{2}}\not\equiv 0\pmod{p}$, the latter following from $0<h(-4p)<p$.
Hence $e\equiv 0\pmod{p}$.
Since there are infinitely many such primes $p$, it follows that $e=0$.
Next, consider the case where $p\equiv 3\pmod{4}$.
Since $0<h(-p)<p$, one has $B_{\frac{p+1}{2}}\not\equiv 0\pmod{p}$, and therefore $b\equiv 0\pmod{p}$.
Again, since there are infinitely many such primes $p$, we obtain $b=0$.
\end{proof}
\section{An intriguing example}\label{sec:pip}
Define $\pi(\bsp)\coloneqq(\pi(p)\bmod{p})_p=(n\bmod{p_n})_{p_n}$, where $p_n$ denotes the $n$th prime.
Although this is somewhat artificial, it is reminiscent of the \emph{Champernowne constant}
\[
0.1234567891011121314151617181920\dots
\]
in the sense that it is formed by listing the natural numbers.
Champernowne proved in \cite{Champernowne1933} that the Champernowne constant is normal in base $10$, and Mahler proved its transcendence in \cite{Mahler1937}.

Since $\pi(p)=O(p^{\epsilon})$ does not hold, Luca--Zudilin's second criterion cannot be applied.
Nevertheless, the irrationality of $\pi(\bsp)$, and even the following stronger statement, can be proved easily.
Although the full strength of the assertion is not needed, we use the prime number theorem.
\begin{proposition}\label{prop:pip}
$\pi(\bsp)\in\calA\setminus\calP^0_{\calA}$.
\end{proposition}
\begin{proof}
Suppose that $\pi(\bsp)\in\mathcal{P}^0_{\mathcal{A}}$.
Then, since $\mathcal{P}^0_{\mathcal{A}}\subset \mathcal{C}_{\mathcal{A}}$, there exists a nonzero polynomial $f(x)$ with rational coefficients such that $f(\pi(\bsp))=0$.
However, by \cite[Theorem~1.4]{Rosen2020}, any such polynomial $f(x)$ must have a rational root.
It follows that there exist an integer $a$ and a positive integer $b$ such that $b\pi(p)-a\equiv 0 \pmod{p}$ holds for infinitely many primes $p$.
Since $\pi(p)\sim \frac{p}{\log p}$, we have $0<b\pi(p)-a<p$ for all sufficiently large $p$, which is a contradiction.
\end{proof}
Although we are not yet able to prove it, we conjecture the following.
\begin{conjecture}\label{conj:pip}
$\pi(\bsp)\in\calA\setminus\calC_{\calA}$.
\end{conjecture}
As evidence for this conjecture, we would like to point out that it can be derived conditionally.
An integer is said to be \emph{$y$-smooth} if all of its prime factors are at most $y$.
The following theorem is known concerning the smoothness of polynomial values.
\begin{theorem}[{Bober--Fretwell--Martin--Wooley \cite[Corollary~1.2]{BoberFretwellMartinWooley2020}}]\label{thm:smooth}
Let $f(x)$ be any quadratic polynomial with integer coefficients.
Then, for every $\varepsilon>0$, there exist infinitely many positive integers $n$ such that $f(n)$ is $n^{\varepsilon}$-smooth.
\end{theorem}
In fact, they derive this theorem by proving a somewhat stronger theorem concerning \emph{polysmoothness} (see \cite[Theorem~1.1]{BoberFretwellMartinWooley2020}).
A generalization of this theorem has not yet been obtained:
\begin{question}[cf.~\cite{BoberFretwellMartinWooley2020}, \cite{Martin02}]\label{que:smooth}
Does \cref{thm:smooth} remain true if $f(x)$ is replaced by an arbitrary nonconstant polynomial in one variable with integer coefficients?
\end{question}
\begin{proposition}\label{prop:condition_pip}
If the answer to \cref{que:smooth} is affirmative, then \cref{conj:pip} is true.
\end{proposition}
\begin{proof}
Let $f(x)$ be a nonconstant polynomial with integer coefficients.
Suppose that $f(n)\equiv 0\pmod{p_n}$ holds for all sufficiently large integer $n$.
Since $p_n\sim n\log n$, it would follow that there exist at most finitely many positive integers $n$ such that $f(n)$ is $n$-smooth.
\end{proof}
Moreover, by using the following result on the equidistribution of roots of quadratic polynomials, one can obtain an unconditional result concerning \cref{conj:pip}.
\begin{theorem}[Duke--Friedlander--Iwaniec \cite{DukeFriedlanderIwaniec1995} and T\'oth \cite{Toth2000}]\label{thm:equidistribution}
Let $f(x)$ be a polynomial with integer coefficients, and let $\alpha$ and $\beta$ satisfy $0\leq \alpha< \beta \le 1$.
Define
\[
R_f(X;\alpha,\beta)\coloneqq\{(p,\nu) : p\leq X,\ f(\nu)\equiv 0\ (\mathrm{mod}\ p),\ \alpha\leq \nu/p<\beta\},
\]
where $p$ is a prime and $\nu$ is an integer satisfying $0\leq \nu \le p-1$.
If $f(x)$ is an irreducible quadratic polynomial, then we have
\[
\#R_f(X;\alpha,\beta)\sim(\beta-\alpha)\pi(X)
\]
as $X\to\infty$.
\end{theorem}
\begin{theorem}
Let $f(x)$ be a nonzero polynomial with integer coefficients such that every irreducible factor of $f(x)$ has degree at most two.
Then $f(\pi(\bsp))\neq 0$.
\end{theorem}
\begin{proof}
Assume that $f(\pi(\bsp))=0$.
Let $X$ be a sufficiently large real number.
Let $l_1(x),\dots,l_a(x)$ be the distinct linear factors of $f(x)$, and let $q_1(x),\dots,q_b(x)$ be the distinct irreducible quadratic factors of $f(x)$.
For a polynomial $g(x)$, define $S_g(X)\coloneqq\{n\leq \pi(X) : p_n\mid g(n)\}$.
Then
\begin{equation}\label{eq:union}
S_f(X)=\bigcup_{i=1}^a S_{l_i}(X)\cup\bigcup_{j=1}^b S_{q_j}(X).
\end{equation}
By the assumption, $S_f(X)$ is in one-to-one correspondence, up to finitely many exceptions independent of $X$, with the set of primes not exceeding $X$.
Therefore,
\begin{equation}\label{eq:S-pi(X)}
\#S_f(X)\sim \pi(X).
\end{equation}
By the same argument as in the proof of \cref{prop:pip}, we have $\#S_{l_i}(X)=O(1)$.
Let $\varepsilon>0$ be small.
Since $\displaystyle\lim_{n\to\infty}n/p_n=0$, if $n\in S_{q_j}(X)$ is sufficiently large, then $(p_n,n)\in R_{q_j}(X;0,\varepsilon)$.
Hence
\[
\#S_{q_j}(X)\leq \#R_{q_j}(X;0,\varepsilon)+O(1).
\]
By \cref{thm:equidistribution}, we have $\#R_{q_j}(X;0,\varepsilon)\sim \varepsilon\pi(X)$, and therefore, by the arbitrariness of $\varepsilon$, it follows that $\#S_{q_j}(X)=o(\pi(X))$.
It then follows from \eqref{eq:union} that
\[
\#S_f(X)=o(\pi(X)),
\]
which contradicts \eqref{eq:S-pi(X)}.
\end{proof}
If an equidistribution result analogous to \cref{thm:equidistribution} could also be established for all irreducible polynomials of degree at least three, then \cref{conj:pip} would be resolved.
Both this plan and the argument of \cref{prop:condition_pip} rely only on the asymptotic behavior of $p_n$.
It would be interesting if \cref{conj:pip} could be proved by making use of more specific properties of $p_n$.
\section{Algebraic independence}\label{sec:alg_indep}
In this section, we discuss algebraic independence over $\bbQ$ of two elements of $\calA$.
Here, for $\alpha,\beta\in\calA$, we say that $\alpha$ and $\beta$ are \emph{algebraically independent over $\bbQ$} if, for every nonzero polynomial $F(x,y)$ in two variables with integer coefficients, one has $F(\alpha,\beta)\neq 0$ in $\calA$. 
 
\begin{lemma}
If integer sequences $(a_p)_p$ and $(b_p)_p$ satisfy $a_p\to\infty$, $b_p\to\infty$, $a_p^d=o(p)$, and $b_p^d=o(a_p)$ as $p\to\infty$ for every positive integer $d$, then $(a_p\bmod{p})_p$ and $(b_p\bmod{p})$ are algebraically independent.
\end{lemma}
\begin{proof}
Let $F(x,y)\in\bbZ[x,y]$ be a nonzero polynomial, and suppose that $F(a_p,b_p)\equiv 0\pmod{p}$ for sufficiently large primes $p$.
Write
\[
F(x,y)=\sum_{i=0}^dg_i(y)x^i,
\]
where $g_i(y)\in\bbZ[y]$ for $0\leq i\leq d$ and $g_d(y)$ is a nonzero polynomial.
Since $b_p\to\infty$, $g_d(b_p)\neq0$ for sufficiently large primes $p$.
In particular, $|g_d(b_p)a_p^d|\geq a_p^d$.
Let $\displaystyle k\coloneqq\max_{0\leq i\leq d}\deg g_i(y)$.
Since $g_i(n)=O(n^k)$ for every $0\leq i\leq d$ and every positive integer $n$, we have
\[
\sum_{i=0}^{d-1}g_i(b_p)a_p^i=O(b_p^ka_p^{d-1})=o(a_p^d)
\]
by $b_p^k=o(a_p)$.
Therefore, $F(a_p,b_p)\neq 0$ for sufficiently large primes $p$.
On the other hand, since $F(a_p,b_p)=O(b_p^ka_p^d)=o(p)$, we have $|F(a_p,b_p)|<p$ for sufficiently large primes $p$.
Thus, we have a contradiction.
\end{proof}
\begin{example}
$(\lfloor\log p\rfloor\bmod{p})_p$ and $(\lfloor\log\log p\rfloor\bmod{p})_p$ are algebraically independent.
\end{example}

\begin{lemma}
Let $\alpha=(a_p)_p\in\calA$ be an element satisfying the conditions of \cref{lem:AnzawaFunakura}.
That is, there exists a strictly increasing sequence of positive integers $(a'_n)_n$ such that, for every $n$, there exist infinitely many primes $p$ such that $a_p = a'_n \bmod{p}$. 
Let $\beta=(b_p)_p\in\calA$ be an element satisfying the conditions of \cref{lem:old}.
That is, there exists a sequence of integers $(b'_p)_p$ such that $b'_p \to \infty$ and $(b'_p)^d = o(p)$ as $p \to \infty$ for every positive integer $d$, and $b_p = b'_p \bmod{p}$.
Then $\alpha$ and $\beta$ are algebraically independent.
\end{lemma}
\begin{proof}
Let $F(x,y)\in\bbZ[x,y]$ be a nonzero polynomial, and suppose that $F(\alpha,\beta)=0$.
Write
\[
F(x,y)=\sum_{i=0}^d f_i(x)y^i,
\]
where $f_i(x)\in\bbZ[x]$ for $0\leq i\leq d$ and $f_d(x)$ is a nonzero polynomial.
For an integer $m$, if $F(m,y)$ is the zero polynomial in $y$, then it must be the case that $f_d(m)=0$.
It follows that there are only finitely many integers $m$ for which $F(m,y)$ is the zero polynomial in $y$.
Therefore, for some $n$, the polynomial $g(y) \coloneqq F(a'_n,y)$ is not the zero polynomial in $y$.
In this case, by the assumption, $g(b'_p)\equiv 0\pmod{p}$ holds for infinitely many primes $p$.
From the assumption on $\beta$, it follows that for such primes $p$, if $p$ is sufficiently large, then $0<|g(b'_p)|<p$, which leads to a contradiction.
\end{proof}
\begin{example}
Anzawa--Funakura's example $(t_{\pi(p)}\bmod{p})_p$ with $(t_n)_n=(1,1,2,1,2,3,1,2,3,4,\dots)$ and $(\lfloor\log p\rfloor\bmod p)_p$ are algebraically independent.
\end{example}

It is also of interest to study algebraic independence for two numbers defined arithmetically.
For example, given two elliptic curves $E_1$ and $E_2$ over $\bbQ$, we may ask whether $\alpha(E_1)$ and $\alpha(E_2)$ are algebraically independent.

By Faltings' theorem, the condition that $E_1$ and $E_2$ are $\bbQ$-isogenous is equivalent to the equality $\alpha(E_1)=\alpha(E_2)$.
Moreover, if $E_2$ is a quadratic twist of $E_1$, then $\alpha(E_1)^2=\alpha(E_2)^2$ holds.
However, apart from such cases of obvious algebraic dependence, it is natural to ask whether $\alpha(E_1)$ and $\alpha(E_2)$ are algebraically independent.
To this end, let us recall the proof that $\alpha(E)$ is transcendental in the naive sense.
The key ingredients in that proof were the counting argument in the proof of Luca--Zudilin's second criterion and the Sato--Tate distribution.

As for the latter ingredient, in order to discuss algebraic independence, it would likely be enough to have the so-called \emph{joint Sato--Tate distribution}, that is, the equidistribution of $(a_p(E_1),a_p(E_2))$ with respect to the product Sato--Tate measure.
Results in this direction have been obtained by Harris~\cite{Harris2009} and Wong~\cite{Wong2019}.

On the other hand, difficulties arise when we try to extend the first ingredient to an argument for algebraic independence.
Indeed, in counting the number of integers that arise in the proof of \cref{lem:LZ-criterion2}, the argument relied on the fact that $a_p(E)$ is of size $O(p^{\frac12})$ by Hasse's theorem, and it worked precisely because $\frac12<1$.
However, if we try to extend this argument naively to pairs, this leads to about $O(p^{\frac12})\times O(p^{\frac12})=O(p)$ possible values, which is too large for the argument to work.

If we restrict to linear polynomials $F(x,y)$, then $F(a_p(E_1), a_p(E_2))$ remains of size $O(p^{\frac12})$, and equidistribution implies its nonvanishing for infinitely many $p$.
In particular, this yields the linear independence of $\alpha(E_1)$ and $\alpha(E_2)$ over $\bbQ$.
\section{Irrationality of values of the logarithmic function}\label{sec:log}
Assuming a certain kind of uniform distribution for $q_p(2)\bmod{p}$, one is led to conjecture that there exist infinitely many \emph{Wieferich primes}, that is, primes $p$ satisfying $q_p(2)\equiv 0\pmod{p}$.
However, only two Wieferich primes, namely $1093$ and $3511$, are currently known, and at present there seems to be no prospect of proving this conjecture.
In the language of $\calA$, this means that $\log_{\calA}(2)$ is conjectured to be a zero divisor.
Combined with \cref{thm:Silverman-0}, this in turn suggests that $\log_{\calA}(2)$ is conjectured to be a nonzero zero divisor, and hence, of course, to be irrational.
Although we are not able to prove that $\log_{\calA}(2)$ is a zero divisor in this paper, we do prove, under the ABC conjecture of Masser and Oesterl\'e, that $\log_{\calA}(2)$ is irrational.
In fact, we prove the same kind of result for $\alpha$-Wieferich primes to bases other than $2$ as well.
It seems unlikely that the criteria in \cref{sec:examples} can be applied to $\log_{\calA}(\alpha)$.
\begin{theorem}\label{thm:unconditional_Silverman}
Let $\alpha\not\in\{+1,-1\}$ be a nonzero rational number.
Then we have $\log_\calA(\alpha)\in\calA\setminus\bbQ^{\times}$.
\end{theorem}
\begin{proof}
Since $-\log_\calA(\alpha^{-1}) = \log_\calA(-\alpha) = \log_\calA \alpha$, there is no loss of generality in assuming that $\alpha > 1$.
Suppose for a contradiction that $\log_\calA(\alpha) = a/b \in \bbQ^\times$ with $a \neq 0$, $b>0$, and $\gcd(a,b) = 1$.
Then for all sufficiently large primes $p \ge N$, we have
\[
q_p(\alpha) \equiv a b^{-1} \pmod{p}.
\]
Write $\alpha = u/v$ with $\mathrm{gcd}(u, v) = 1$.
Note that $u>v>0$.
Let $\ell$ be a sufficiently large prime (such that $\ell > u, |a|, b, N$) and consider the value of the homogenization of the $\ell$th cyclotomic polynomial
\[
\Phi_\ell(u,v) \coloneqq v^{\ell-1} \frac{\alpha^\ell-1}{\alpha -1} = \frac{u^\ell - v^\ell}{u-v} = u^{\ell-1} + u^{\ell-2} v + \cdots + v^{\ell-1} \in \bbZ.
\]
First, by the lifting-the-exponent lemma, any prime divisor $p$ of $u-v$ does not divide $\Phi_\ell(u,v)$, since $p \nmid \ell$.
Let $p$ be any prime divisor of $\Phi_\ell(u,v)$.
Since $u$ and $v$ are coprime, we have $p\nmid v$.
Hence $\alpha^\ell\equiv 1 \pmod{p}$ and $\alpha \not\equiv 1 \pmod{p}$, so that $\ord_p(\alpha)=\ell$.
In particular, such a prime $p$ satisfies $p\equiv 1 \pmod{\ell}$, and therefore $p>\ell>N$.

We next show that $\Phi_\ell(u,v)$ is square-free.
For each prime $p\mid\Phi_\ell(u,v)$, write
\[
\Phi_\ell(u,v) = p t_p.
\]
From $\alpha^\ell = 1 + v^{1-\ell} (\alpha-1) pt_p$, we obtain
\[
\ell q_p(\alpha) \equiv q_p(\alpha^\ell) = \frac{(1 + v^{1-\ell} (\alpha-1) pt_p)^{p-1} - 1}{p} \equiv - v^{-\ell} (u-v) t_p \pmod{p}.
\]
By our assumption, we get
\begin{equation}\label{eq:contra-cong}
(u-v) b t_p + a \ell v^{\ell} \equiv 0 \pmod{p}.
\end{equation}
If $p \mid t_p$, then $a\ell v^\ell \equiv 0 \pmod{p}$, which is impossible since $p > \ell$ and $\ell$ is sufficiently large (note that $a\neq 0$).
Thus $p \nmid t_p$, and hence $\Phi_\ell(u,v)$ is square-free.

Now we define
\begin{equation}\label{def:Tell}
T_\ell\coloneqq (u-v) b \sum_{p \mid \Phi_\ell(u,v)} \frac{\Phi_\ell(u,v)}{p} + a\ell v^\ell \in \bbZ.
\end{equation}
Then \eqref{eq:contra-cong} implies that for each prime $p\mid\Phi_\ell(u,v)$, $T_\ell\equiv0\pmod{p}$.
Since $\Phi_\ell(u,v)$ is square-free, it follows that
\begin{equation}\label{eq:Phi<T}
\Phi_\ell(u,v) \mid T_\ell.
\end{equation}

Since $\Phi_\ell(u,v)$ is square-free, we can write $\Phi_\ell(u,v) = p_1 \cdots p_r$ with distinct primes $p_1 < \cdots < p_r$.
As observed above, each $p_j \equiv 1 \pmod{\ell}$, so we may write $p_j = \ell m_j+1$ with $m_1 < \cdots < m_r$.
In particular, $m_j \ge j$.
Since $p_j > \ell$, we have $\ell^r<\Phi_\ell(u,v)<\ell u^{\ell-1}<\ell^\ell$, which implies $1 \le r < \ell$.
Moreover,
\[
\sum_{p\mid\Phi_\ell(u,v)} \frac{1}{p} = \sum_{j=1}^r \frac{1}{\ell m_j +1} \le \sum_{j=1}^r \frac{1}{\ell j+1} < \frac{H_r}{\ell} < \frac{1+\log \ell}{\ell},
\]
where $H_n = \sum_{j=1}^n 1/j$ is the $n$th harmonic number.
Hence, for sufficiently large $\ell$, since $\alpha > 1$, we obtain
\[
\sum_{p\mid\Phi_\ell(u,v)}\frac{1}{p}<\frac{1}{2(u-v)b} \quad \text{and} \quad 2|a|<\frac{\Phi_\ell(u,v)}{\ell v^\ell},
\]
which implies
\begin{equation}\label{eq:T<Phi}
|T_\ell|<\Phi_\ell(u,v).
\end{equation}
Combining \eqref{eq:Phi<T} and \eqref{eq:T<Phi}, we conclude that $T_\ell = 0$.

By the definition of $T_\ell$ in~\eqref{def:Tell}, this yields
\[
(u-v)b\sum_{p\mid\Phi_\ell(u,v)}\frac{\Phi_\ell(u,v)}{p}=-a\ell v^\ell.
\]
For each $p_i$, since $p_i\equiv 1 \pmod{\ell}$, we have
\[
t_{p_i}=\frac{\Phi_{\ell}(u,v)}{p_i}=\prod_{1\le j\leq r, \ j\neq i} p_j \equiv 1\pmod{\ell}.
\]
Therefore, by reducing modulo $\ell$, we have
\[
(u-v)br\equiv 0 \pmod{\ell}.
\]
However, by construction, the prime $\ell$ is chosen to be larger than both $u-v>0$ and $b>0$, and moreover $1\leq r < \ell$, which yields a contradiction.
\end{proof}
When $\alpha=2$, only five primes are currently known to satisfy $q_p(2) \equiv 1 \pmod{p}$, that is,
\[
2^{p-1}\equiv 1+p \pmod{p^2},
\]
namely $3$, $29$, $37$, $3373$, and $2001907169$ (see OEIS \href{https://oeis.org/A125854}{A125854}). 
By analogy with non-Wieferich primes, this naturally suggests that there are infinitely many primes with $q_p(2) \not\equiv 1 \pmod{p}$.
The special case $\log_\calA(2)\neq 1$ of this theorem means that this infinitude holds unconditionally.

Furthermore, together with~\cref{thm:Silverman-0}, we immediately obtain the following result.
\begin{corollary}\label{cor:Silverman-irr}
Let $\alpha\not\in\{+1,-1\}$ be a nonzero rational number.
Then, $\log_\calA(\alpha) \in \calA\setminus\bbQ$ under the ABC conjecture.
\end{corollary}
Let $(G_n(x))_n$ be the sequence of \emph{Gregory polynomials} defined by
\[
\frac{t(1+t)^x}{\log(1+t)} = \sum_{n=0}^\infty G_n(x) t^n.
\]
In particular, when $x=0$, the coefficients $G_n\coloneqq G_n(0)$ are called the \emph{Gregory coefficients}.
The Gregory coefficients are another analogue of the Bernoulli numbers, distinct from the Euler numbers; indeed, they are also sometimes called the \emph{Bernoulli numbers of the second kind}.
For a rational number $x$ and an integer $k \ge 2$, set $G_\calA(k; x) \coloneqq (G_{p-k}(x) \bmod{p})_p\in\calA$.
Then, the above corollary immediately strengthens \cite[Theorem~3.3]{KanekoMatsusakaSeki2025}.
\begin{corollary}
Let $k \ge 2$ be an integer and $x\in\bbZ\setminus[-k-1,-1]$.
Then $G_\calA(k;x)\in\calA\setminus\bbQ$ under the ABC conjecture.
\end{corollary}
\begin{proof}
As a generalization of \cite[Theorem 3.2]{KanekoMatsusakaSeki2025}, we have
\begin{align*}
G_\calA(k;x) &= (-1)^{k-1} \sum_{j=0}^k (-1)^j \binom{k}{j} (x+j+1) \log_\calA(x+j+1)\\
&= (-1)^{k-1} \log_\calA(h_k(x))
\end{align*}
by \cite[Theorem 3.8]{MatsusakaMiyazakiYara2026}.
Here, we set
\[
h_k(x)\coloneqq\prod_{j=0}^k (x+j+1)^{(-1)^j (x+j+1)\binom{k}{j}}.
\]
We can check that $h_k(-k-2-x)=h_k(x)^{(-1)^{k+1}}$ and if $x\geq 0$ then $h_k(x)$ is rational and
\[
\log(h_k(x))=(k-2)!\int_{[0,1]^k}\frac{\mathrm{d}t_1\cdots\mathrm{d}t_k}{(x+1+t_1+\cdots+t_k)^{k-1}}>0.
\]
Therefore, the claim follows from \cref{cor:Silverman-irr}.
\end{proof}
In particular, $G_{\calA}(k)\coloneqq G_{\calA}(k;0)=(G_{p-k} \bmod{p})_p$ is an analogue of $Z_{\calA}(k)=(B_{p-k}/k \bmod{p})_p$, and it is interesting that, in the case of $G_{\calA}(k)$, one can obtain results on irrationality in this way.

\bibliography{references}

@article {Wong2019,
    AUTHOR = {Wong, Peng-Jie},
     TITLE = {On the {C}hebotarev-{S}ato-{T}ate phenomenon},
   JOURNAL = {J. Number Theory},
  FJOURNAL = {Journal of Number Theory},
    VOLUME = {196},
      YEAR = {2019},
     PAGES = {272--290}
}

@incollection {Harris2009,
    AUTHOR = {Harris, Michael},
     TITLE = {Potential automorphy of odd-dimensional symmetric powers of
              elliptic curves and applications},
 BOOKTITLE = {Algebra, arithmetic, and geometry: in honor of {Y}u. {I}.
              {M}anin. {V}ol. {II}},
    SERIES = {Progr. Math.},
    VOLUME = {270},
     PAGES = {1--21},
 PUBLISHER = {Birkh\"{a}user Boston, Boston, MA},
      YEAR = {2009}
}

@misc{MatsusakaMiyazakiYara2026,
      title={On finite analogues of {D}obi\'{n}ski's formula and of {E}uler's constant via {G}regory polynomials}, 
      author={Matsusaka, Toshiki and Miyazaki, Taichi and Yara, Shunta},
      eprint={2604.01578},
      note={\href{https://arxiv.org/abs/2604.01578}{arXiv:2604.01578}}
}

@incollection {Sutherland2019,
    AUTHOR = {Sutherland, Andrew V.},
     TITLE = {Sato-{T}ate distributions},
 BOOKTITLE = {Analytic methods in arithmetic geometry},
    SERIES = {Contemp. Math.},
    VOLUME = {740},
     PAGES = {197--248},
 PUBLISHER = {Amer. Math. Soc., [Providence], RI},
      YEAR = {2019}
}

@article {BGHT2011,
    AUTHOR = {Barnet-Lamb, Tom and Geraghty, David and Harris, Michael and
              Taylor, Richard},
     TITLE = {A family of {C}alabi-{Y}au varieties and potential automorphy
              {II}},
   JOURNAL = {Publ. Res. Inst. Math. Sci.},
  FJOURNAL = {Publications of the Research Institute for Mathematical
              Sciences},
    VOLUME = {47},
      YEAR = {2011},
    NUMBER = {1},
     PAGES = {29--98}
}

@book {Sills2018,
    AUTHOR = {Sills, A. V.},
     TITLE = {An invitation to the {R}ogers-{R}amanujan identities},
      NOTE = {With a foreword by George E. Andrews},
 PUBLISHER = {CRC Press, Boca Raton, FL},
      YEAR = {2018},
     PAGES = {xx+233},
      ISBN = {978-1-4987-4525-3},
   MRCLASS = {11P84 (05A17 05A19 05A30 33D05 33D15)},
MRREVIEWER = {Jeremy Lovejoy},
}

@article {Bressoud1981,
    AUTHOR = {Bressoud, D. M.},
     TITLE = {Some identities for terminating {$q$}-series},
   JOURNAL = {Math. Proc. Cambridge Philos. Soc.},
  FJOURNAL = {Mathematical Proceedings of the Cambridge Philosophical
              Society},
    VOLUME = {89},
      YEAR = {1981},
    NUMBER = {2},
     PAGES = {211--223},
}

@article {AnzawaFunakura2024,
    AUTHOR = {Anzawa, Takumi and Funakura, Hidetaka},
     TITLE = {Congruences for the {$q$}-{F}ibonacci sequence related to its
              transcendence},
   JOURNAL = {Ramanujan J.},
  FJOURNAL = {Ramanujan Journal. An International Journal Devoted to the
              Areas of Mathematics Influenced by Ramanujan},
    VOLUME = {63},
      YEAR = {2024},
    NUMBER = {4},
     PAGES = {1057--1072},
      ISSN = {1382-4090,1572-9303},
   MRCLASS = {11B39 (11M32)},
MRREVIEWER = {Rusen\ Li},
       DOI = {10.1007/s11139-023-00802-5},
       URL = {https://doi.org/10.1007/s11139-023-00802-5},
}

@article {LucaZudilin2025,
    AUTHOR = {Luca, Florian and Zudilin, Wadim},
     TITLE = {Irrationality and transcendence questions in the `poor man's
              ad\`ele ring'},
   JOURNAL = {Ramanujan J.},
  FJOURNAL = {Ramanujan Journal. An International Journal Devoted to the
              Areas of Mathematics Influenced by Ramanujan},
    VOLUME = {67},
      YEAR = {2025},
    NUMBER = {4},
     PAGES = {Paper No. 88, 10}
}

@article {Rosen2020,
    AUTHOR = {Rosen, Julian},
     TITLE = {A finite analogue of the ring of algebraic numbers},
   JOURNAL = {J. Number Theory},
  FJOURNAL = {Journal of Number Theory},
    VOLUME = {208},
      YEAR = {2020},
     PAGES = {59--71}
}

@article{KanekoZagier,
      title={Finite multiple zeta values}, 
      author={Kaneko, M. and Zagier, D.},
      year = {to appear},
      note={Advanced Studies in Pure Math.}
}

@article {LucaZudilin,
    AUTHOR = {Luca, Florian and Zudilin, Wadim},
     TITLE = {Poor man's transcendence for {F}robenius traces of elliptic curves},
     year = {to appear},
      note = {Advanced Studies in Pure Math.}
}

@article {Ax1968,
    AUTHOR = {Ax, James},
     TITLE = {The elementary theory of finite fields},
   JOURNAL = {Ann. of Math. (2)},
  FJOURNAL = {Annals of Mathematics. Second Series},
    VOLUME = {88},
      YEAR = {1968},
     PAGES = {239--271},
      ISSN = {0003-486X},
   MRCLASS = {10.80},
MRREVIEWER = {D.\ J.\ Lewis},
       DOI = {10.2307/1970573},
       URL = {https://doi.org/10.2307/1970573},
}

@article {Jarossay2020,
    AUTHOR = {Jarossay, David},
     TITLE = {Depth reductions for associators},
   JOURNAL = {J. Number Theory},
  FJOURNAL = {Journal of Number Theory},
    VOLUME = {217},
      YEAR = {2020},
     PAGES = {163--192},
      ISSN = {0022-314X,1096-1658},
   MRCLASS = {11M32 (11S82 14G32)},
MRREVIEWER = {Nils\ Matthes},
       DOI = {10.1016/j.jnt.2020.04.019},
       URL = {https://doi.org/10.1016/j.jnt.2020.04.019},
}

@article{Schur1917,
  author = {Schur, Issai},
  title = {Ein {B}eitrag zur additiven {Z}ahlentheorie und zur {T}heorie der {K}ettenbr{\"u}che},
  journal = {Sitzungsberichte der Preu{\ss}ischen Akademie der Wissenschaften, Physikalisch-Mathematische Klasse},
  year = {1917},
  pages = {302--321}
}

@misc{Rosen,
  author = {Rosen, Julian},
  title = {Sequential periods of the crystalline {F}robenius},
  eprint = {1805.01885},
  archivePrefix = {arXiv},
  primaryClass = {math.NT},
  note = {\href{https://arxiv.org/abs/1805.01885}{arXiv:1805.01885}},
  year = {unpublished}
}

@article {BoberFretwellMartinWooley2020,
    AUTHOR = {Bober, J. W. and Fretwell, D. and Martin, G. and Wooley, T.
              D.},
     TITLE = {Smooth values of polynomials},
   JOURNAL = {J. Aust. Math. Soc.},
  FJOURNAL = {Journal of the Australian Mathematical Society},
    VOLUME = {108},
      YEAR = {2020},
    NUMBER = {2},
     PAGES = {245--261},
      ISSN = {1446-7887,1446-8107},
   MRCLASS = {11N32 (11N25 12E05)},
MRREVIEWER = {Mario\ Pineda-Ruelas},
       DOI = {10.1017/s1446788718000320},
       URL = {https://doi.org/10.1017/s1446788718000320},
}

@article {Seki2024,
    AUTHOR = {Seki, S.},
     TITLE = {Regular primes, non-{W}ieferich primes, and finite multiple
              zeta values of level {$N$}},
   JOURNAL = {Integers},
  FJOURNAL = {Integers. Electronic Journal of Combinatorial Number Theory},
    VOLUME = {24},
      YEAR = {2024},
     PAGES = {Paper No. A22, 14},
      ISSN = {1553-1732},
   MRCLASS = {11M32 (11A41)},
MRREVIEWER = {Lee-Peng\ Teo},
}

@article {KanekoMatsusakaSeki2025,
    AUTHOR = {Kaneko, Masanobu and Matsusaka, Toshiki and Seki, S.},
     TITLE = {On finite analogues of {E}uler's constant},
   JOURNAL = {Int. Math. Res. Not. IMRN},
  FJOURNAL = {International Mathematics Research Notices. IMRN},
      YEAR = {2025},
    NUMBER = {2},
     PAGES = {Paper No. rnae281, 12},
      ISSN = {1073-7928,1687-0247},
   MRCLASS = {11M32 (11B68)},
MRREVIEWER = {Lee-Peng\ Teo},
       DOI = {10.1093/imrn/rnae281},
       URL = {https://doi.org/10.1093/imrn/rnae281},
}

@article {Silverman1988,
    AUTHOR = {Silverman, Joseph H.},
     TITLE = {Wieferich's criterion and the {$abc$}-conjecture},
   JOURNAL = {J. Number Theory},
  FJOURNAL = {Journal of Number Theory},
    VOLUME = {30},
      YEAR = {1988},
    NUMBER = {2},
     PAGES = {226--237},
      ISSN = {0022-314X,1096-1658},
   MRCLASS = {11D41 (11G05)},
MRREVIEWER = {Daniel\ Bertrand},
       DOI = {10.1016/0022-314X(88)90019-4},
       URL = {https://doi.org/10.1016/0022-314X(88)90019-4},
}

@article {Carlitz1954,
    AUTHOR = {Carlitz, L.},
     TITLE = {The class number of an imaginary quadratic field},
   JOURNAL = {Comment. Math. Helv.},
  FJOURNAL = {Commentarii Mathematici Helvetici},
    VOLUME = {27},
      YEAR = {1953},
     PAGES = {338--345 (1954)},
      ISSN = {0010-2571,1420-8946},
   MRCLASS = {10.0X},
MRREVIEWER = {A.\ L.\ Whiteman},
       DOI = {10.1007/BF02564567},
       URL = {https://doi.org/10.1007/BF02564567},
}

@article{Siegel1935,
  author  = {Siegel, Carl Ludwig},
  title   = {{\"U}ber die {C}lassenzahl quadratischer {Z}ahlk{\"o}rper},
  journal = {Acta Arithmetica},
  volume  = {1},
  year    = {1935},
  pages   = {83--86},
  doi     = {10.4064/aa-1-1-83-86}
}

@article {Vandiver1940,
    AUTHOR = {Vandiver, H. S.},
     TITLE = {Note on {E}uler number criteria for the first case of
              {F}ermat's last theorem},
   JOURNAL = {Amer. J. Math.},
  FJOURNAL = {American Journal of Mathematics},
    VOLUME = {62},
      YEAR = {1940},
     PAGES = {79--82},
      ISSN = {0002-9327,1080-6377},
   MRCLASS = {10.0X},
MRREVIEWER = {N.\ G. W. H. Beeger},
       DOI = {10.2307/2371437},
       URL = {https://doi.org/10.2307/2371437},
}

@misc{MatsuzukiSakamotoUeki,
  author = {Matsuzuki, Daichi and Sakamoto, Honami and Ueki, Jun},
  title = {Positive characteristic analogues of finite algebraic numbers},
  eprint = {2601.21209},
  archivePrefix = {arXiv},
  primaryClass = {math.NT},
  note={\href{https://arxiv.org/abs/2601.21209}{arXiv:2601.21209}}
}

@article {Martin02,
    AUTHOR = {Martin, Greg},
     TITLE = {An asymptotic formula for the number of smooth values of a
              polynomial},
   JOURNAL = {J. Number Theory},
  FJOURNAL = {Journal of Number Theory},
    VOLUME = {93},
      YEAR = {2002},
    NUMBER = {2},
     PAGES = {108--182},
      ISSN = {0022-314X,1096-1658},
   MRCLASS = {11N32 (11Y05 11Y11)},
MRREVIEWER = {G.\ Greaves},
       DOI = {10.1006/jnth.2001.2722},
       URL = {https://doi.org/10.1006/jnth.2001.2722},
}

@article {DukeFriedlanderIwaniec1995,
    AUTHOR = {Duke, W. and Friedlander, J. B. and Iwaniec, H.},
     TITLE = {Equidistribution of roots of a quadratic congruence to prime
              moduli},
   JOURNAL = {Ann. of Math. (2)},
  FJOURNAL = {Annals of Mathematics. Second Series},
    VOLUME = {141},
      YEAR = {1995},
    NUMBER = {2},
     PAGES = {423--441},
      ISSN = {0003-486X,1939-8980},
   MRCLASS = {11N64 (11L20)},
MRREVIEWER = {D.\ R.\ Heath-Brown},
       DOI = {10.2307/2118527},
       URL = {https://doi.org/10.2307/2118527},
}

@article {Toth2000,
    AUTHOR = {T\'{o}th, \'{A}rp\'{a}d},
     TITLE = {Roots of quadratic congruences},
   JOURNAL = {Internat. Math. Res. Notices},
  FJOURNAL = {International Mathematics Research Notices},
      YEAR = {2000},
    NUMBER = {14},
     PAGES = {719--739},
      ISSN = {1073-7928,1687-0247},
   MRCLASS = {11N32 (11L05 11L20 11N36)},
MRREVIEWER = {John\ B.\ Friedlander},
       DOI = {10.1155/S1073792800000404},
       URL = {https://doi.org/10.1155/S1073792800000404},
}

@article {Champernowne1933,
    AUTHOR = {Champernowne, D. G.},
     TITLE = {The Construction of Decimals Normal in the Scale of
              Ten},
   JOURNAL = {J. London Math. Soc.},
  FJOURNAL = {The Journal of the London Mathematical Society},
    VOLUME = {8},
      YEAR = {1933},
    NUMBER = {4},
     PAGES = {254--260},
      ISSN = {0024-6107,1469-7750},
   MRCLASS = {DML},
       DOI = {10.1112/jlms/s1-8.4.254},
       URL = {https://doi.org/10.1112/jlms/s1-8.4.254},
}

@article {Ford2008,
    AUTHOR = {Ford, Kevin},
     TITLE = {The distribution of integers with a divisor in a given
              interval},
   JOURNAL = {Ann. of Math. (2)},
  FJOURNAL = {Annals of Mathematics. Second Series},
    VOLUME = {168},
      YEAR = {2008},
    NUMBER = {2},
     PAGES = {367--433},
      ISSN = {0003-486X,1939-8980},
   MRCLASS = {11N25 (11N37)},
MRREVIEWER = {D.\ R.\ Heath-Brown},
       DOI = {10.4007/annals.2008.168.367},
       URL = {https://doi.org/10.4007/annals.2008.168.367},
}

@article{Mahler1937,
  author  = {Mahler, Kurt},
  title   = {Arithmetische {E}igenschaften einer {K}lasse von {D}ezimalbr{\"u}chen},
  journal = {Proceedings of the Koninklijke Nederlandse Akademie van Wetenschappen. Series A},
  volume  = {40},
  year    = {1937},
  pages   = {421--428}
}

\end{document}